\documentclass[12pt]{amsart}

\usepackage{amsmath,amssymb,amsbsy,amsfonts,amsthm,latexsym,
                     amsopn,amstext,amsxtra,euscript,amscd,mathrsfs}
\usepackage{amsmath,amssymb,amsbsy,amsfonts,latexsym,amsopn,amstext,cite,
                                               amsxtra,euscript,amscd,bm}
\usepackage{graphics,epsfig}
\usepackage{graphicx}

\usepackage{mathtools}
\usepackage{todonotes}
\usepackage{url}
\usepackage[colorlinks,linkcolor=blue,anchorcolor=blue,citecolor=blue,backref=page]{hyperref}
\usepackage{multirow}
\usepackage{float}
\usepackage{color}

\usepackage[np]{numprint}
\npdecimalsign{\ensuremath{.}}

\usepackage{bibentry}

\usepackage[norefs,nocites]{refcheck}

\renewcommand*{\backref}[1]{}
\renewcommand*{\backrefalt}[4]{%
    \ifcase #1 (Not cited.)%
    \or        (p.\,#2)%
    \else      (pp.\,#2)%
    \fi}

\begin{document}

\newtheorem{theorem}{Theorem}
\newtheorem{lemma}[theorem]{Lemma}
\newtheorem{claim}[theorem]{Claim}
\newtheorem{cor}[theorem]{Corollary}
\newtheorem{prop}[theorem]{Proposition}
\newtheorem{definition}{Definition}
\newtheorem{question}[theorem]{Question}
\newtheorem{remark}[theorem]{Remark}
\newcommand{\hh}{{{\mathrm h}}}

\numberwithin{equation}{section}
\numberwithin{theorem}{section}
\numberwithin{table}{section}
\numberwithin{figure}{section}
%% \numberwithin{remark}{section}

\def\sssum{\mathop{\sum\!\sum\!\sum}}
\def\ssum{\mathop{\sum\ldots \sum}}
\def\iint{\mathop{\int\ldots \int}}

\def\ges{\gtrsim}
\def\les{\lesssim}

\def\sA{\mathscr A}
\def\sB{\mathscr B}
\def\sC{\mathscr C}
\def\sD{\Delta}
\def\sE{\mathscr E}
\def\sF{\mathscr F}
\def\sG{\mathscr G}
\def\sH{\mathscr H}
\def\sI{\mathscr I}
\def\sJ{\mathscr J}
\def\sK{\mathscr K}
\def\sL{\mathscr L}
\def\sM{\mathscr M}
\def\sN{\mathscr N}
\def\sO{\mathscr O}
\def\sP{\mathscr P}
\def\sQ{\mathscr Q}
\def\sR{\mathscr R}
\def\sS{\mathscr S}
\def\sU{\mathscr U}
\def\sT{\mathscr T}
\def\sV{\mathscr V}
\def\sW{\mathscr W}
\def\sX{\mathscr X}
\def\sY{\mathscr Y}
\def\sZ{\mathscr Z}

\def\squareforqed{\hbox{\rlap{$\sqcap$}$\sqcup$}}
\def\qed{\ifmmode\squareforqed\else{\unskip\nobreak\hfil
\penalty50\hskip1em\null\nobreak\hfil\squareforqed
\parfillskip=0pt\finalhyphendemerits=0\endgraf}\fi}%%

%  use the AMS-Euler Fraktur fonts
%%%%%%%%%%%%%%%%%%%%%%%%%%%%%%%%%%
\newfont{\teneufm}{eufm10}
\newfont{\seveneufm}{eufm7}
\newfont{\fiveeufm}{eufm5}
%%%%%%%%%%%%%%%%%%%%%%%%%%%%%%%%%
%
%  allow automatic size selection in math mode
%
%%%%%%%%%%%%%%%%%%%%%%%%%%%%%%%%%
%\newfam\eufmfam
%     \textfont\eufmfam=\teneufm
%\scriptfont\eufmfam=\seveneufm
%     \scriptscriptfont\eufmfam=\fiveeufm
%%%%%%%%%%%%%%%%%%%%%%%%%%%%%%%%%
%
%  \frak works on a single symbol at a time...
%
\def\frak#1{{\fam\eufmfam\relax#1}}

\newcommand{\bflambda}{{\boldsymbol{\lambda}}}
\newcommand{\bfmu}{{\boldsymbol{\mu}}}
\newcommand{\bfxi}{{\boldsymbol{\xi}}}
\newcommand{\bfrho}{{\boldsymbol{\rho}}}

\newcommand{\bfalpha}{{\boldsymbol{\alpha}}}
\newcommand{\bfbeta}{{\boldsymbol{\beta}}}
\newcommand{\bfphi}{{\boldsymbol{\varphi}}}
\newcommand{\bfpsi}{{\boldsymbol{\psi}}}
\newcommand{\bftheta}{{\boldsymbol{\vartheta}}}

\def\fK{Frak K}
\def\fT{Frak{T}}

\def\fA{{Frak A}}
\def\fB{{Frak B}}
\def\fC{\mathfrak{C}}

\def \balpha{\bm{\alpha}}
\def \bbeta{\bm{\beta}}
\def \bgamma{\bm{\gamma}}
\def \blambda{\bm{\lambda}}
\def \bchi{\bm{\chi}}
\def \bphi{\bm{\varphi}}
\def \bpsi{\bm{\psi}}

\def\eqref#1{(\ref{#1})}

\def\vec#1{\mathbf{#1}}

%\def\squareforqed{\hbox{\rlap{$\sqcap$}$\sqcup$}}
%\def\qed{\ifmmode\squareforqed\else{\unskip\nobreak\hfil
%\penalty50\hskip1em\null\nobreak\hfil\squareforqed
%\parfillskip=0pt\finalhyphendemerits=0\endgraf}\fi}

%%%%%%%%%%%%%%%%%%%%%%%%%
% Alphabet calligraphie %
%%%%%%%%%%%%%%%%%%%%%%%%%
\def\cA{{\mathcal A}}
\def\cB{{\mathcal B}}
\def\cC{{\mathcal C}}
\def\cD{{\mathcal D}}
\def\cE{{\mathcal E}}
\def\cF{{\mathcal F}}
\def\cG{{\mathcal G}}
\def\cH{{\mathcal H}}
\def\cI{{\mathcal I}}
\def\cJ{{\mathcal J}}
\def\cK{{\mathcal K}}
\def\cL{{\mathcal L}}
\def\cM{{\mathcal M}}
\def\cN{{\mathcal N}}
\def\cO{{\mathcal O}}
\def\cP{{\mathcal P}}
\def\cQ{{\mathcal Q}}
\def\cR{{\mathcal R}}
\def\cS{{\mathcal S}}
\def\cT{{\mathcal T}}
\def\cU{{\mathcal U}}
\def\cV{{\mathcal V}}
\def\cW{{\mathcal W}}
\def\cX{{\mathcal X}}
\def\cY{{\mathcal Y}}
\def\cZ{{\mathcal Z}}
\newcommand{\rmod}[1]{\: \mbox{mod} \: #1}

\def\cg{{\mathcal g}}

\def\vr{\mathbf r}

\def\e{{\mathbf{\,e}}}
\def\ep{{\mathbf{\,e}}_p}
\def\em{{\mathbf{\,e}}_m}

\def\Tr{{\mathrm{Tr}}}
\def\Nm{{\mathrm{Nm}}}

 \def\SS{{\mathbf{S}}}

\def\lcm{{\mathrm{lcm}}}
\def\ord{{\mathrm{ord}}}

\def\({\left(}
\def\){\right)}
\def\fl#1{\left\lfloor#1\right\rfloor}
\def\rf#1{\left\lceil#1\right\rceil}

\def\mand{\qquad \mbox{and} \qquad}

%\newcommand{\commK}[1]{\marginpar{%
%\begin{color}{red}
%\vskip-\baselineskip %raise the marginpar a bit
%\raggedright\footnotesize
%\itshape\hrule \smallskip K: #1\par\smallskip\hrule\end{color}}}
%
%\newcommand{\commI}[1]{\marginpar{%
%\begin{color}{magenta}
%\vskip-\baselineskip %raise the marginpar a bit
%\raggedright\footnotesize
%\itshape\hrule \smallskip I: #1\par\smallskip\hrule\end{color}}}
%
%
%\newcommand{\comms}[1]{\marginpar{%
%\begin{color}{blue}
%\vskip-\baselineskip %raise the marginpar a bit
%\raggedright\footnotesize
%\itshape\hrule \smallskip S: #1\par\smallskip\hrule\end{color}}}

%%%%%%%%%%%%%%%%%%%%%%%%%%%%%%%%%%%%%%%%%%%%%%%%%%%%%%%%
%%%%%%%%%%%%%%%%%%%%%%%%%%%%%%%%%%%%%%%%%%%%%%%%%%%%%%%%
%%%%%%%%%%%%%%%%%%%%%%%%%%%%%%%%%%%%%%%%%%%%%%%%%%%%%%%%
%%%%%%%%%%%%%%%%%%%%%%%%%%%%%%%%%%%%%%%%%%%%%%%%%%%%%%%%

%%%%%%%  END OF STANDARD STUFF %%%%%%%%%

%%%%%%%%%%%%%%%%%%%%%%%%%%%%%%%%%%%%%%%%%%%%%%%%%%%%%%%%
%%%%%%%%%%%%%%%%%%%%%%%%%%%%%%%%%%%%%%%%%%%%%%%%%%%%%%%%
%%%%%%%%%%%%%%%%%%%%%%%%%%%%%%%%%%%%%%%%%%%%%%%%%%%%%%%%
%%%%%%%%%%%%%%%%%%%%%%%%%%%%%%%%%%%%%%%%%%%%%%%%%%%%%%%
%%%%%%%%%%%
%%% Spell

\hyphenation{re-pub-lished}

\mathsurround=1pt

\def\bfdefault{b}
\overfullrule=5pt

\def \F{{\mathbb F}}
\def \K{{\mathbb K}}
\def \N{{\mathbb N}}
\def \Z{{\mathbb Z}}
\def \Q{{\mathbb Q}}
\def \R{{\mathbb R}}
\def \C{{\mathbb C}}
\def\Fp{\F_p}
\def \fp{\Fp^*}

\def\Kmnp{\cK_p(m,n)}
\def\Kmnq{\cK_q(m,n)}
\def\Kmnp{\cK_p(m,n)}
\def\Kxmnq{\cK_q(\bfxi; m,n)}
\def\Kxmnp{\cK_p(\bfxi; m,n)}
\def\Kxnumnp{\cK_{\nu,p}(\bfxi; m,n)}
\def\Kxnumnq{\cK_{\nu,q}(\bfxi; m,n)}

\def\Kmn{\cK_p(m,n)}
\def\psmn{\psi_p(m,n)}

\def \xbar{\overline x}
\def\e{{\mathbf{\,e}}}
\def\ep{{\mathbf{\,e}}_p}
\def\eq{{\mathbf{\,e}}_q}

\title[Exponential sums with sparse polynomials]{Exponential sums with sparse polynomials over finite fields}
\author{Igor Shparlinski and Qiang Wang}

\address{School of Mathematics and Statistics, 
The University of New South Wales, 
Sydney NSW 2052, 
Australia}
\email{igor.shparlinski@unsw.edu.au}

\address{School of Mathematics and Statistics, 
Carleton University, 
1125 Colonel By Drive, 
Ottawa, ON K1S 5B6, 
Canada}
\email{wang@math.carleton.ca}

\begin{abstract} We obtain new bounds of exponential sums modulo a prime $p$  with sparse polynomials  $a_0x^{n_0} + \cdots  + a_{\nu}x^{n_\nu}$. The bounds depend on  various greatest common divisors of exponents $n_0, \ldots, n_\nu$ and  their  differences. In particular, two new bounds for binomials are obtained, improving previous results in broad ranges of parameters. 
 \end{abstract}


\maketitle

\section{Introduction}

For a prime $p$,  positive integers $n_0, \ldots, n_\nu$ and arbitrary integer coefficients $a_0, \ldots, a_\nu$,  
we consider the exponential sum
\[
S_{n_0, \ldots, n_\nu}(a_0, \ldots, a_\nu) = \sum_{x=0}^{p-1} \ep(a_0 x^{n_0} + \cdots + a_\nu x^{n_\nu}),
\]
where $\ep(x) = e^{2\pi i x/p}$.

For the convenience, we denote 
$$
M_{n_0, \ldots, n_\nu} = \max_{\substack{a_0, \ldots, a_\nu \in \Z\\ \gcd\(\prod_{i=0}^k a_i,p\)=1}} \left|S_{n_0, \ldots, n_\nu}(a_0, \ldots, a_\nu)\right|.
$$

The classical Weil bound on  exponential sums with general polynomials,  see, for example,~\cite[Theorem~5.38]{LN}, 
implies that if at least one of the coefficients does not vanish modulo $p$ then 
\begin{equation}
\label{eq:Weil}
|  S_{n_0, \ldots, n_\nu}(a_0, \ldots, a_\nu) | \le \max\{n_0, \ldots, n_\nu\} p^{1/2}. 
\end{equation}

Clearly, the bound~\eqref{eq:Weil} becomes trivial if $ \max\{n_0, \ldots, n_\nu\} \ge p^{1/2}$. Thus starting from work of  Akulinichev~\cite{Aku}, there have been a chain of consistent efforts to obtain nontrivial bounds beyond this restriction, see~\cite{CCP1,CCP2, CP05, CP10, Mac, MPSS, MSS} and references therein.

The case   of binomial sums (that is, $\nu=1$) has always been of special interest~\cite{CP03, CP11, Kar, SV20}.  
In particular, bounds for  
binomial sums have played a key role in  resolution of {\it Goresky--Klapper conjecture}, see~\cite{GK}, and 
its generalisation~\cite{ACMPPRT,BCPP1,CoKo,CP11,GKMS,CMPR} and in the closely related generalised 
{\it Lehmer conjecture\/}~\cite{BCPP2}. 
 
In fact in the case of binomial sums $S_{m, n}(a,b)$ one is especially interested in $m=1$ which is important for the above applications. 

 For example, if $m=1$ Akulinichev~\cite[Theorem~1]{Aku} has given the bound 
\begin{equation}
\label{eq:Aku}
M_{1, n}  \le p/\sqrt{\gcd(n, p-1)}, 
\end{equation}
and then  combining~\eqref{eq:Aku}  with~\eqref{eq:Weil}  shown that for $n \mid p-1$ the
following uniform bound holds 
\begin{equation}
\label{eq:Aku 56}
M_{1, n}  \le p^{5/6},
\end{equation}
see~\cite[Corollary]{Aku}. The bound~\eqref{eq:Aku 56} has been improved in~\cite[Corollary~3.3]{SV20}
as follows
\begin{equation}
\label{eq:SV 45}
M_{1, n}  = O\(p^{4/5}\),
\end{equation}
(with an absolute implied constant).

In fact, the bound~\eqref{eq:SV 45} is based on an improvement of a bound of 
Karatsuba~\cite[Theorem~1]{Kar}
$$
M_{1, n} \le (n-1)^{1/4} p^{3/4}, 
$$
which by~\cite[Theorem~3.2]{SV20} can be replaced with 
\begin{equation}
\label{eq:ShpVol}
M_{1, n}  \le p^{3/4} + (n-1)^{1/3} p^{2/3}. 
\end{equation}

We also recall the  following bound of  Cochrane and Pinner~\cite[Theorem~1.3]{CP11} 
\begin{equation}\label{eq:CP11}
M_{m, n}  \le  d + 2.292 e^{13/46} p^{89/92}, 
\end{equation}  
where $d =    \gcd(n-m, p-1)$ and $e=   \gcd(m, n, p-1)$.

 The bounds of sums with more monomials usually involve more parameters and conditions
 and are somewhat too technical to survey and compare here. However we believe that our bounds
expand the class of general sparse polynomials which admit nontrivial bounds and they certainly do
so in the case of binomial sums.

\section{Main Results}

We  recall that  the notations $A=O(B)$, $A\ll B$ and $B \gg A$ are each equivalent to the
statement that the inequality $A\le c\,B$ holds with a
constant $c>0$ which is absolute throughout this paper. 

In what follows, it  is convenient to introduce notation $A\les B$ and $B\ges A$ as equivalents
of $A \le p^{o(1)} B$ as $p\to \infty$. 

\begin{theorem}
\label{thm:Bound de} For $\nu\ge 2$   positive integers $n_0, n_1, \ldots, n_\nu$,  we denote 
$$ d = \gcd(n_1 - n_0, \ldots, n_{\nu}-n_{0}, p-1),   \qquad e  = \gcd(d, n_0),  .
$$
and 
$$
D =\min_{0\leq i \leq \nu} \max_{j\neq i} \gcd(n_j- n_i, p-1), \qquad \Gamma = (p-1)/D,  \qquad  \Delta = d/e.
$$
Then 
$$
M_{n_0, \ldots, n_\nu} \les
  \begin{cases} 
 \Delta^{-1/4}\Gamma^{-1/4\nu}p^{7/6}, &\quad  \text{if $p^{29/48} \le \Delta <p^{2/3}$,}\\
 \Delta^{-21/52}\Gamma^{-1/4\nu}p^{131/104}, &\quad  \text{if $p^{59/112} \le  \Delta <p^{29/48}$,}\\
 \Delta^{-7/20}\Gamma^{-1/4\nu}p^{197/160}, &\quad  \text{if $p^{1/2} \le  \Delta <p^{59/112}$,}\\
 \Delta^{-31/80}\Gamma^{-1/4\nu}p^{5/4}, & \quad \text{if $ \Delta< p^{1/2}$.}
\end{cases}   
$$  
\end{theorem}

The above bound is nontrivial when
$$ \max\{p^{29/48}, \Gamma^{-1/\nu} p^{2/3}\}<  \Delta < p^{2/3}, 
$$ 
or 
$$ \max\{p^{59/112}, \Gamma^{-13/21\nu} p^{27/42}\} <  \Delta < p^{29/48},
$$
or 
$$ \max\{p^{1/2}, \Gamma^{-5/7\nu} p^{37/56}\} <  \Delta < p^{59/112},
$$
or 
$$ \Gamma^{-20/31\nu} p^{20/31} <  \Delta < p^{1/2}.
$$

When $\nu=1$, it is easy to see that   $D=d$ and
$$
e = \gcd( \gcd(n -m, p-1), m) = \gcd(n -m, m, p-1) = \gcd(m,n, p-1) 
$$
and thus for binomials   the following result holds.

\begin{cor} 
\label{cor:Bound de} For positive integers $m$ and $n$,  we denote 
$$ d=  \gcd(n-m, p-1) \mand  e =  \gcd(m,n, p-1).
$$
Then 
$$
M_{m, n} \les
  \begin{cases} 
e^{1/4}p^{11/12} , &\quad  \text{if $p^{29/48} \le d/e<p^{2/3}$,}\\
e^{21/52}d^{-2/13}p^{105/104}, &\quad  \text{if $p^{59/112} \le d/e <p^{29/48}$,}\\
e^{7/20}d^{-1/10}p^{157/160},  &\quad  \text{if $p^{1/2} \le d/e <p^{59/112}$,}\\
e^{31/80}d^{-11/80}p, & \quad \text{if $ d/e< p^{1/2}$.}
\end{cases}   
$$
\end{cor}

 In the case when  integers $m$ and $n$  satisfy  
$$
n \mid p-1 \mand \gcd(m,n)=1
$$
Akulinichev~\cite[Theorem~3]{Aku} has shown that 
\begin{equation}\label{eq:A65}
M_{m, n} \le pn^{-1} + h^{1/2} p^{3/4}
\end{equation}
where $h = \gcd(m,p-1)$.  
Using the same idea  as in the proof of Theorem~\ref{thm:Bound de}  we obtain another bound in terms of $h$ and $m$ which improve previous bounds
in some other cases.

\begin{theorem} \label{thm:Bound ell}
For positive integers $m$ and $n$ such that 
$$
n \mid p-1 \mand \gcd(m, n)=1 \mand  h = \gcd(m,p-1).
$$
Then  
$$
M_{m, n} \les
  \begin{cases} 
h^{1/4}n^{-1/4} p^{11/12}, &\quad  \text{if $p^{29/48} \le n <p^{2/3}$,}\\
h^{1/4}n^{-21/52}p^{105/104}, &\quad  \text{if $p^{59/112} \le n<p^{29/48}$,}\\
h^{1/4}n^{-7/20}p^{157/160}, &\quad  \text{if $p^{1/2} \le n<p^{59/112}$,}\\
h^{1/4}n^{-31/80}p, & \quad \text{if $ n< p^{1/2}$.}
\end{cases}   
$$
\end{theorem}

It is difficult to give an exact region when Theorem~\ref{thm:Bound de}  
improves the large variety of previous results more sparse polynomials. However, in Section~\ref{sec:comp} we compare Corollary~\ref{cor:Bound de}
and Theorem~\ref{thm:Bound ell}  
 with previous bounds~~\eqref{eq:CP11} and~\eqref{eq:A65}
for binomial sums, which also depend on various greatest common divisors (while the bounds~\eqref{eq:Weil} 
and~\eqref{eq:ShpVol}  depend on the size of the exponents and so are incomparable with our result

%
%It is difficult to give an exact region when Theorems~\ref{thm:Bound de} and~\ref{thm:Bound ell}
%improve the large variety of previous results. However, in Section~\ref{sec:comp} we compare Corollary~\ref{cor:Bound de}
%and Theorem~\ref{thm:Bound ell} between themselves  \commS{we said in Section~\ref{sec:comp}  that they are incomparable?}
%and also with previous bounds~~\eqref{eq:CP11} and~\eqref{eq:A65}
%for binomial sums, which also depend on various greatest common divisors (while the bounds~\eqref{eq:Weil} 
%and~\eqref{eq:ShpVol}  depend on the size of the exponents and so are incomparable with our results). 

\section{Preparations}
\label{sec:prelim}

For a positive integer $t \mid p-1$ we use $T_t$ to denote the number of solutions to the equation
$$
u^t + v^t \equiv x^t + y^t \pmod p, \qquad 1 \le u,v,x,y < p.
$$

We estimate $T_t$ via a  combination of recent results of  Shkredov~\cite[Theorem~8]{Shkr1}
and of Murphy,  Rudnev,  Shkredov, and 
Shteinikov~\cite[Theorems~1.4 and~6.5]{MRSS}  (which in turn improve the previous result of  
Heath-Brown and  Konyagin~\cite[Lemma~3]{HBK}) on  the {\it additive energy of multiplicative subgroups\/}, that is, 
on the number of solutions $E(\Gamma)$ to the equations
$$
u_1+u_2 = v_1+v_2, \qquad u_1,u_2 , v_1, v_2 \in \Gamma,
$$
where $\Gamma$ is a multiplicative subgroup of $\F_p^*$.

We also note that unfortunately there is a misprint in the formulation of~\cite[Theorem~8]{Shkr1}
(some terms ought be commuted and the symbol `$\min$' ought to stay in a different place). More specifically, in the notation 
of~\cite{Shkr1} the 
bound given by~\cite[Theorem~8]{Shkr1},  is of the form 
\begin{align*} 
E(\Gamma) & \le  \(\# \Gamma\)^3 p^{-1/3} \log \# \Gamma \\
&\qquad  \quad + \min\left \{ p^{1/26} \(\# \Gamma\)^{31/13} \log^{8/13}\# \Gamma, \(\# \Gamma\)^{32/13} \log^{41/65} \# \Gamma\right \},
\end{align*}
where 
$$
E(\Gamma) = \#\left\{\(u,v,x,y\) \in \Gamma^4:~u+v= x+y \right\}
$$
is the  additive energy of $\Gamma$. 

Taking the above into account, we derive the following bound on $T_t$, where we   have suppressed some logarithmic 
factors via the use of  the symbol `$\les$'.

\begin{lemma}
\label{lem:Mt} 	
We have 
$$
T_t \les
\begin{cases} 
 p^{8/3} t, &\quad  \text{if $p^{29/48} \le (p-1)/t <p^{2/3}$,}\\
  p^{63/26} t^{21/13}, &\quad  \text{if $p^{59/112} \le (p-1)/t <p^{29/48}$,}\\
 % \max\{ p^{101/40} t^{7/5}, p^{63/24} t^{7/6} \}  &\quad  \text{if $p^{1/2} \le (p-1)/t <p^{4/7}$,}\\
 p^{101/40} t^{7/5},  &\quad  \text{if $p^{1/2} \le (p-1)/t <p^{59/112}$,}\\
 p^{49/20} t^{31/20}, & \quad \text{if $ (p-1)/t < p^{1/2}$.}
\end{cases} 
$$
\end{lemma}  

We also need recall the following bound given by \cite[Lemma~7]{CFKLLS}.

\begin{lemma}
\label{lem:SprEq Zeros}
For
$\nu +1\ge 2$
elements $a_0, a_1, \ldots\,, a_\nu \in \F_p^*$ and
arbitrary 
integers $t_0,  t_1, \ldots , t_\nu<p$,
the number of solutions  $Q$  to the equation
$$
\sum_{i=0}^\nu a_ix^{t_i} = 0, \qquad x \in \F_p^*,
$$
with $t_0 = 0$, satisfies
$$
Q \le 2  p^{1 - 1/\nu} D^{1/\nu} + O(p^{1 - 2/\nu} D^{2/\nu}),
$$
where
$$
D = \min_{0 \le i \le \nu} \max_{j \ne i} \gcd(t_j - t_i, p-1).
$$
\end{lemma}

We now derive the following estimate for points on sparse curves. 

\begin{lemma}
\label{lem:SprCurve Zeros}
For
$\nu \ge 1$
elements $a_0, a_1, \ldots\,, a_\nu \in \F_p^*$ and arbitrary 
integers $t_0,  t_1, \ldots , t_\nu<p$,
the number of solutions  $R$  to the equation
$$
\sum_{i=0}^\nu a_ix^{t_i} = \sum_{i=0}^\nu a_iy^{t_i} , \qquad x,y \in \F_p^*,
$$
with $t_0 = 0$, satisfies
$$
R \ll   p^{2 - 1/\nu} D^{1/\nu} ,
$$
where 
$$
D = \min_{0 \le i \le \nu} \max_{j \ne i} \gcd(t_j - t_i, p-1).
$$
\end{lemma}

\begin{proof}
We now  write $x = yz$ and obtain  
\begin{align*}
R
 &
 =  \#\{(y,z) \in\(\Fp^*\)^2:~   a_0(z^{t_0}-1) + a_1 y^{t_1-t_0} (z^{t_1} -1) +\cdots \\
&\qquad \qquad \qquad \qquad \qquad \qquad \qquad \qquad   + 
a_\nu  y^{t_\nu-t_0} (z^{t_\nu} -1)= 0 \}. 
\end{align*}
If 
$$
z^{t_0}-1=z^{t_1} -1= \ldots = z^{t_\nu} -1= 0
$$
then obviously $z^d =1$ where 
$$
 d = \gcd(t_1 - t_0, \ldots, t_{\nu}-t_{0}, p-1)
$$
 and thus there are at most $d$ such values of $z$, for which there are at most $p$ values of $y$.
Otherwise, for $\nu \ge 1$, by Lemma~\ref{lem:SprEq Zeros} we have for each $z$ we have at most $O( p^{1 - 1/\nu} D^{1/\nu})$ values of $y$.  Hence $R \ll dp + p^{1 - 1/\nu} D^{1/\nu} $. Because obviously  
$d \mid t_i - t_j$ for all $1\le i < j \le \nu$, we have $d \le  D$ and thus $dp\le Dp  \ll  p^{2 - 1/\nu} D^{1/\nu}$ 
which implies the desired bound. 
\end{proof}

\section{Proof of Theorem~\ref{thm:Bound de}}

%\[
%d = \gcd(n_1 - n_0, \ldots, n_{k}-n_{0}, p-1)~ \text{and }   s = \frac{p-1}{d}.
%\]  
%
%Since $n_0 \mid p-1$,  we have
%\begin{eqnarray*}
% e &=&  \gcd(n_\nu,  \ldots, n_0) \\
%    &= &   \gcd(n_\nu, \ldots, n_0, p-1)\\
%    &  =&  \gcd(n_\nu-n_0, \ldots, n_1-n_0,  n_0, p-1) \\
%     & = &  \gcd(d, n_0).
%\end{eqnarray*}

We fix some $a_0, \ldots, a_\nu$ with 
$$
\gcd\(a_0\ldots a_\nu, p\)=1
$$ 
and consider the sum
$$
S^* =  \sum_{x\in\Fp^*}  \ep\(x^{n_0}\(a_0 +  a_1 x^{e_1 d} + \cdots a_\nu x^{e_\nu d}\)\)
$$ 
over the multiplicative group $\Fp^*$ of the finite field of $p$ elements.  
Clearly it is enough to estimate the sum $S^*$.

Let  $n_i -n_{0} = e_i d$  for $1\leq i \leq \nu$, thus 
\[
a_0  + a_1 x^{n_1-n_0} + \cdots + a_\nu x^{n_\nu-n_0} = a_0 +  a_1 x^{e_1 d} + \cdots a_\nu x^{e_\nu d}.
\]

Let 
\[
s = \frac{p-1}{d}.
\]

Now, using that for any $y\in\Fp^*$ we have $y^{sd} = 1$, we derive
\begin{align*}
&S^*  =  \sum_{x\in\Fp^*} \ep\(x^{n_0}\(a_0 +  a_1 x^{e_1 d} + \cdots a_\nu x^{e_\nu d}\)\) \\
& = \frac{1}{p-1} \sum_{y\in\Fp^*} \sum_{x\in\Fp^*} 
 \ep\(\(xy^s\)^{n_0}\(a_0 +  a_1 (xy^s)^{e_1 d} + \cdots a_\nu (xy^s)^{e_\nu d}\)\) \\
&  =  \frac{1}{p-1} \sum_{x\in\Fp} \sum_{y\in\Fp^*} \ep\(x^{n_0} y^{n_0 s} \(a_0  + a_1 x^{n_1-n_0} + \cdots + a_\nu x^{n_\nu-n_0} \)\) .
\end{align*}

Let 
$$
N(\lambda) = \#\{x\in\Fp^*:~ a_0x^{n_0}  + a_1 x^{n_1} + \cdots + a_\nu x^{n_\nu}= \lambda\},
$$ 
thus we can write
$$
S^* =  \frac{1}{p-1} \sum_{\lambda \in\Fp} N(\lambda)   \sum_{y\in\Fp^*} \ep(\lambda  y^{n_0 s}) .
$$

By the H{\"o}lder inequality 
\begin{equation}
\label{eq:Hold}
|S^*|^4 \ll   p^{-3}  \left(\sum_{\lambda \in\Fp}  N(\lambda) \right)^2 \sum_{\lambda \in\Fp} N(\lambda)^2
 \sum_{\lambda \in\Fp} \left|  \sum_{y\in\Fp^*} \ep(\lambda  y^{n_0 s})\right|^4. 
\end{equation} 

By the  orthogonality of exponential functions
\begin{align*}
\sum_{\lambda \in\Fp} & \left|  \sum_{y\in\Fp^*} \ep(\lambda  y^{n_0 s})\right|^4\\
 & \quad =
\sum_{\lambda \in\Fp}    \sum_{u,v, y,z\in\Fp^*}
\ep\(\lambda  \(u^{n_0 s} + v^{n_0 s} - y^{n_0 s}   -z^{n_0 s}\)\)\\
& \quad =
 \sum_{u,v, y,z\in\Fp^*} \sum_{\lambda \in\Fp}   
\ep\(\lambda  \(u^{n_0 s} + v^{n_0 s} - y^{n_0 s}   -z^{n_0 s}\)\) = p  T_{ n_0  s},
 \end{align*} 
where $T_t$ is defined as in Section~\ref{sec:prelim}. 

Let 
$$r=\gcd(n_0 s, p-1) = s \gcd\(n_0, (p-1)/s\) = s \gcd\(n_0,d\) = es.
$$ 
Then we have $T_{n_0  s} = T_r$.  Hence we rewrite~\eqref{eq:Hold} as 
\begin{equation}
\label{eq:S NN2Mld}
\(S^*\)^4 \ll   p^{-3}  \left(\sum_{\lambda \in\Fp}  N(\lambda) \right)^2 \sum_{\lambda \in\Fp} N(\lambda)^2 T_r. 
\end{equation}

Trivially, we have 
\begin{equation}
\label{eq:sum N}
\sum_{\lambda \in\Fp}  N(\lambda) = p.
\end{equation}

Furthermore, we have
\begin{align*}
\sum_{\lambda \in\Fp}  & N(\lambda)^2 \\
& =  \#\{(x,y) \in\(\Fp^*\)^2:~  a_0x^{n_0}  + \cdots + a_\nu x^{n_\nu}= a_0y^{n_0}    + \cdots + a_\nu y^{n_\nu}\}.
\end{align*}
%%We now  write $x = yz$ and obtain 
%%\begin{align*}
%%\sum_{\lambda \in\Fp}  & N(\lambda)^2 \\
%% &
%% =  \#\{(y,z) \in\(\Fp^*\)^2:~   a_0(z^{n_0}-1) + a_1 y^{n_1-n_0} (z^{n_1} -1) +\cdots \\
%%&\qquad \qquad \qquad \qquad \qquad \qquad \qquad \qquad   + 
%%a_\nu  y^{n_\nu-n_0} (z^{n_\nu} -1)= 0 \}. 
%%\end{align*}
%%If 
%%$$
%%z^{n_0}-1=z^{n_1} -1= \ldots = z^{n_\nu} -1= 0
%%$$
%%then obviously $z^d =1$ and thus there are at most $d$ such values of $z$, for which there are at most $p$ values of $y$.
%%Otherwise, by Lemma~\ref{lem:SprEq Zeros} we have for each $z$ we have at most $O( p^{1 - 1/\nu} D^{1/\nu})$ values of $y$. 
Therefore, by Lemma~\ref{lem:SprCurve Zeros} 
\begin{equation}
\label{eq:sum N2}
\sum_{\lambda \in\Fp}  N(\lambda)^2 \ll  p^{2 - 1/\nu} D^{1/\nu}.
\end{equation}

Note that 
$$
\frac{p-1}{r}  = \frac{p-1}{es}  = \frac{d}{e} .
$$
Hence Lemma~\ref{lem:Mt} 	implies 
\begin{equation}
\label{eq:Mr}
T_r \les
\begin{cases} 
 (e/d)  p^{11/3}, &\quad  \text{if $p^{29/48} \le d/e <p^{2/3}$,}\\
(e/d)^{21/13}   p^{105/26} , &\quad  \text{if $p^{59/112} \le d/e<p^{29/48}$,}\\
 (e/d)^{7/5}  p^{157/40}  &\quad  \text{if $p^{1/2} \le d/e <p^{59/112}$,}\\
(e/d)^{31/20}  p^{4} , & \quad \text{if $ d/e < p^{1/2}$.}
\end{cases} 
\end{equation} 
%\commI{Please check the cut-offs and the bounds. The we need to calculate our final estimates} 
Hence substituting the bounds~\eqref{eq:sum N}, \eqref{eq:sum N2} and~\eqref{eq:Mr}
in~\eqref{eq:S NN2Mld} we conclude the proof.

\section{Proof of Theorem~\ref{thm:Bound ell}}

We fix some integers $a$ and $b$ with $\gcd(ab,p)=1$ and denote 
$$
S^*=   \sum_{x\in\Fp^*}   \ep\(ax^m + bx^n\).
$$

Denoting $s = (p-1)/n$ we obtain 
\begin{align*}
S^* & = \frac{1}{p-1}  \sum_{y\in\Fp^*}   \sum_{x\in\Fp^*}   \ep\(a\(xy^s\)^m + b\(xy^s\)^n\)\\
& =  \frac{1}{p-1}  \sum_{x\in\Fp^*}   \sum_{y\in\Fp^*}  \ep\(ax^m  y^{ms}   + bx^n\). 
\end{align*}
Since $\gcd(m,n)=1$ we can replace $y^{ms}$ with just $y^s$ as both functions run
through the same set of elements of $\F_p^*$ of order $n$ and take each value exactly $s$ times.
Hence 
$$
S^* =  \frac{1}{p-1}  \sum_{x\in\Fp^*}   \ep\(bx^n\)   \sum_{y\in\Fp^*}  \ep\(ax^my^{s}\).
$$
By the Cauchy inequality we have  
\begin{align*}
|S^*|^2 &  \le   \frac{1}{p}  \sum_{x\in\Fp^*}  \left|  \sum_{y\in\Fp^*}  \ep\(bx^my^{s}\)\right|^2 \\
&  = \frac{1}{p}  \sum_{y,z\in\Fp^*}  \sum_{x\in\Fp^*}  \ep\(ax^m\(y^{s}-z^s\)\)
=  \frac{1}{p}  \sum_{\lambda\in\Fp} R(\lambda)  \sum_{x\in\Fp^*}  \ep\(b\lambda x^m\)  , 
\end{align*}
where 
$$
R(\lambda) = \#\{(y,z) \in\(\Fp^*\)^2:~ y^{s}-z^s= \lambda\}.
$$ 
Clearly 
$$
 \sum_{\lambda\in\Fp} R(\lambda) ^2 = T_s, 
$$
where $T_s$ is as in Lemma~\ref{lem:Mt}. 
Therefore,  applying  the Cauchy inequality  one more time, and using the orthogonality 
of exponential functions, we obtain  
\begin{align*}
|S^*|^4 &  \ll   \frac{1}{p^2} T_s   \sum_{\lambda\in\Fp} \left| \sum_{x\in\Fp^*}  \ep\(b\lambda x^n\) \right|^2
\\
& =  \frac{1}{p^2} T_s p\#\left\{(u,v)\in \(\Fp^*\)^2:~ u^m = v^m\right \} 
= \frac{1}{p^2} T_s  hp^2 = h T_s . 
\end{align*}
Using Lemma~\ref{lem:Mt} we obtain the desired result.

\section{Comparison}
\label{sec:comp}

We first note that   bounds~\eqref{eq:Weil} and~\eqref{eq:ShpVol},  Corollary~\ref{cor:Bound de} and Theorem~\ref{thm:Bound ell}
all hold  and nontrivial under different conditions and thus are not directly comparable.

Hence we only compare  Corollary~\ref{cor:Bound de} with   the bound~\eqref{eq:CP11} and 
then also Theorem~\ref{thm:Bound ell} with  bounds~\eqref{eq:CP11} and~\eqref{eq:A65}. 
Note that when we compare Theorem~\ref{thm:Bound ell} with~\eqref{eq:CP11} we always have $e =1$ 
(since $\gcd(m,n)=1$) and we also assume that $d \le p^{89/92}$.

For example, the  bound of  Corollary~\ref{cor:Bound de} improves previous known bounds when $e$ is small
but $d$ is sufficiently large.

Indeed, 
when $p^{29/48} \le d/e<p^{2/3}$, because $p^{11/12} < p^{89/92}$ and $e$ is  sufficiently  small comparing to $p$
 our bound is clearly better than~\eqref{eq:CP11}. 
 
 When  $ p^{59/112}  \le d/e < p^{29/48}$,  then using $d \ge e p^{59/112}$
 we obtain   
$$e^{21/52}d^{-2/13}p^{105/104} \le    e^{1/4} p^{13/14}, 
 $$ 
which is    better than~\eqref{eq:CP11}  for a small $e$.

When $ p^{1/2}  \le d/e < p^{59/112}$,   then using $d \ge e p^{1/2}$ we have 
$$e^{7/20} d^{-1/10} p^{157/160}  \le  e^{1/4}p^{149/160},
$$ 
which is  also  better than~\eqref{eq:CP11}  for a small $e$.

Finally for $d \le p^{1/2}$ and $e = O(1)$, our bound improves~\eqref{eq:CP11} for $d > p^{60/253}$.

To demonstrate the strength of our result, we let  $d=p^{\delta+o(1)}$  and $e= p^{\gamma+o(1)}$. Then Figure~\ref{fig0} ($x$-axis is $\delta$ and $y$-axis is $\gamma$)  shows the regions of $(\delta, \gamma)$ where the bounds of  Corollary~\ref{thm:Bound de} are better than  both~\eqref{eq:CP11} and nontrivial.

\begin{figure}[H]
\centering 
\includegraphics[width=4.8in]{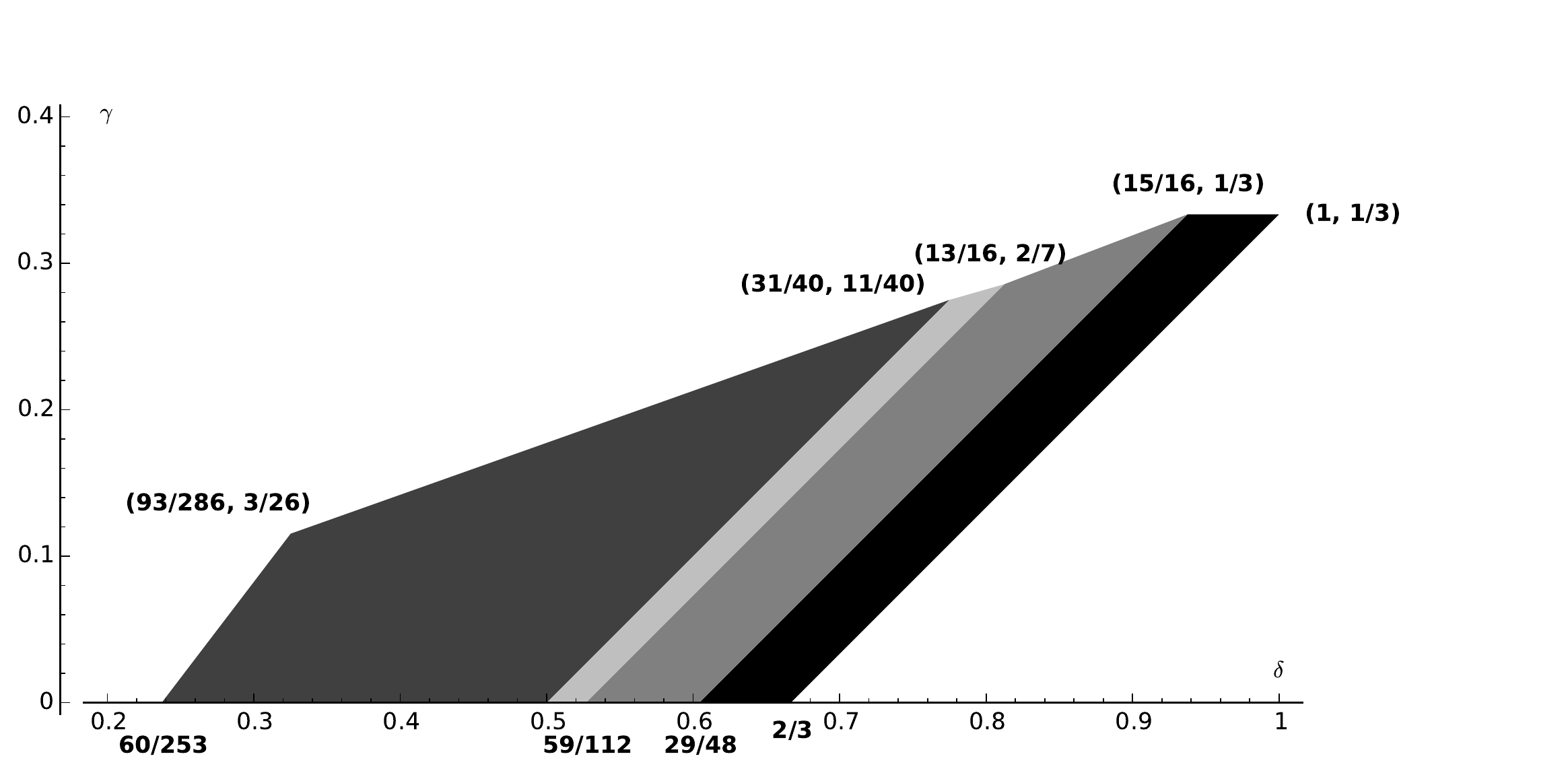}
\caption{Polygon of  Corollary~\ref{cor:Bound de} improving~\eqref{eq:CP11} with $d=p^{\delta+o(1)}$  and $e= p^{\gamma +o(1)}$}
\label{fig0}
\end{figure} % \commI{Changed the with to 4.8in as it preserves the same number of pages} 

It is easy to check that one can produce infinite series of examples with parameters arbitrary close to any point 
 inside of the polygon of  Figure~\ref{fig0}. 

Next, we note that Theorem~\ref{thm:Bound ell} gives nontrivial bounds in any one of the following cases.  
\begin{itemize}
\item[(i)] $\max\{p^{29/48},  hp^{-1/3}\} \le n < p^{2/3}$,  
\item[(ii)] $\max\{p^{59/112}, h^{13/21} p^{1/42}\} \le n< p^{29/48}$, 
\item[(iii)] $\max\{p^{1/2},   h^{5/7} p^{-3/56}\} \le n< p^{59/112}$,         
\item[(iv)]  $h^{20/31} \le n < p^{1/2}$. 
\end{itemize}

In  Case~(i),  if also  $hn > p^{2/3}$, 
then $h^{1/4} n^{-1/4} p < h^{1/2} p^{3/4}$ and we  improve~\eqref{eq:A65}. 
If also  $n > h p^{3/23}$, then $h^{1/4}n^{-1/4} p < p^{89/92}$ and we also  improve~\eqref{eq:CP11}.

In  Case~(ii), if also   
  $h n^{21/13} >  p^{27/26}$, then  $h^{1/4} n^{-21/52}  p^{105/104} $  $<  h^{1/2}  p^{3/4}$ and we  thus improve~\eqref{eq:A65}. 
If also $n > h^{13/21}p^{101/966}$, then we have
$h^{1/4}n^{-21/52}   p^{105/104}   < p^{89/92}$ and thus we   
improve~\eqref{eq:CP11} as well. 
  
 In  Case~(iii), if also  
  $h n^{7/5} >  p^{37/40}$, then  $h^{1/4} n^{-7/20}  p^{157/160} <  h^{1/2}  p^{3/4}$ and we  thus improve~\eqref{eq:A65}. 
If also $n > h^{5/7}p^{51/1288}$, then  we have
$h^{1/4}n^{-7/20} p^{157/160}  < p^{89/92}$ and thus we   
improve~\eqref{eq:CP11} as well.

In   Case~(iv),  if also   $hn^{31/20} > p$, then   we  improve~\eqref{eq:A65} 
 because   $h^{1/4} n^{-31/80} p < h^{1/2} p^{3/4}$.  If also $n > h^{20/31}  p^{60/713}$ , then $h^{1/4} n^{-31/80} p < p^{89/92}$ and thus we 
 also improve~\eqref{eq:CP11}.

 It is easy to see that in each of the Cases~(i), (ii),  (iii) and (iv) the above ranges of $h$ and $m$ overlap so in each of them we 
 sometimes improve both~\eqref{eq:CP11} and~\eqref{eq:A65} simultaneously. 

To demonstrate our result, we let $h=p^{\varepsilon+o(1)}$ and $n= p^{\eta+o(1)}$. Then Figute~\ref{fig1} ($x$-axis is $\varepsilon$ and $y$-axis is $\eta$)  shows the regions of $(\varepsilon, \eta)$ where the bounds of  Theorem~\ref{thm:Bound ell} are better than 
both~\eqref{eq:CP11} and~\eqref{eq:A65}.   We want to emphasize that $hn < p$ because $\gcd(m,n)=1$ implies that  $h=\gcd(m, p-1) = \gcd(m, p-1) \leq (p-1)/n$. 

\begin{figure}[H]
\centering 
\includegraphics[width=4.8in]{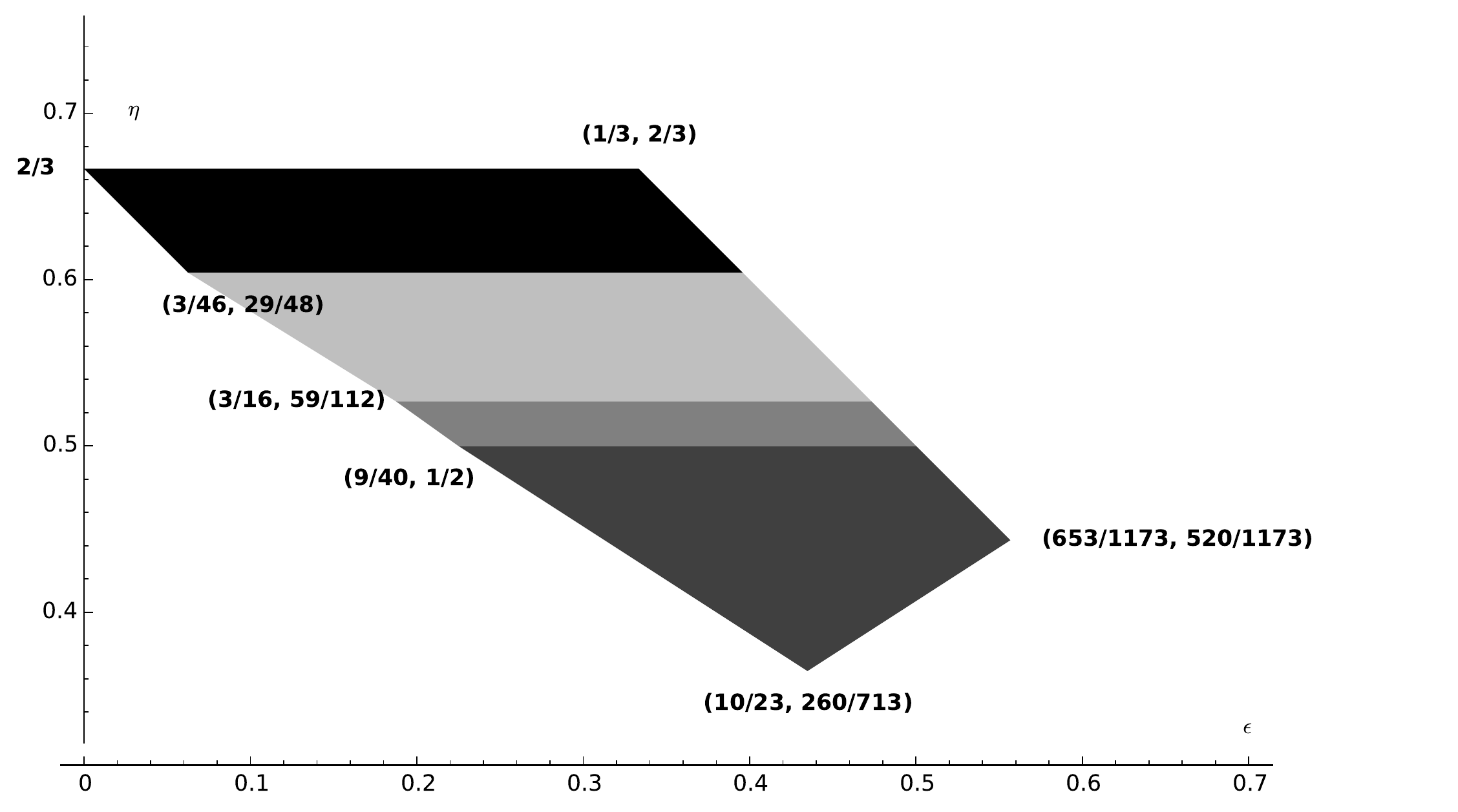}
\caption{Polygon of Theorem~\ref{thm:Bound ell} improving~\eqref{eq:CP11} and~\eqref{eq:A65} with $h= p^{\varepsilon+o(1)}$   and $n= p^{\eta+o(1)}$}
\label{fig1}
\end{figure}

As before, we note that is easy to check that one can produce infinite series of examples with parameters arbitrary close to any point 
 inside of the polygon of  Figure~\ref{fig1}.   
 
 \section{Comments}

 One can certainly use higher powers of sums $S^*$ in the proofs of both Theorems~~\ref{thm:Bound de} and~\ref{thm:Bound ell}.
 For example, for any integer $\nu\ge 1$ we can generalise~\eqref{eq:S NN2Mld} as 
$$
\(S^*\)^{2\nu}  \ll   p^{-2\nu +1} \(\sum_{\lambda \in\Fp}  N(\lambda) \)^{2\nu-2} \sum_{\lambda \in\Fp} N(\lambda)^2 T_{\nu,r}, 
$$
where $T_{\nu,r}$ is the number of solutions to the equation
$$
u_1^t + \ldots + u_\nu^t \equiv v_1^t + \ldots + v_\nu^t  \pmod p, \quad 1 \le u_1, v_1, \ldots , u_\nu ,v_\nu < p.
$$ 
Hence we need analogues of Lemma~\ref{lem:Mt} 	for  $T_{\nu,r}$. Such nontrivial bounds on $T_{\nu,r}$
are indeed available, for example they can be derived from~\cite[Lemma~4.4]{MRSS} for $\nu =3$ and~\cite[Theorems~3 and~25]{Shkr2}
for larger values of $\nu$. 
%\commS{it seems that it is only for $\nu=2^k$. Maybe we change ``arbitraray'' to ``higher order''.}
%\commI{In principle it was correct as stated by convexity, but it's better re-reword indeed. See how I phrased  this now.} 
However with the present knowledge of such bounds it is not clear whether one can  obtain better bounds of exponential 
sums.

\section*{Acknowledgement}

The authors are grateful to Todd Cochrane  for very  helpful suggestions.
The authors would like to thank Ilya Shkredov for clarification of the results of~\cite{Shkr1}.

The authors is grateful to the organisers of the   14th International Conference
on Finite Fields and their Applications, $\F_{q^{14}}$,  12--14 June, 2019, Vancouver, 
where this work started.

During the preparation of this work, I.S  was supported  by Australian Research Council  Grant DP170100786
and by the Natural Science Foundation of China Grant~11871317.  Q. W. was supported by NSERC of Canada  (RGPIN-2017-06410).

\end{document}